\documentclass[11 pt]{amsart}
\usepackage {amssymb,latexsym,amsthm,amsmath,amsfonts, mathtools, setspace, enumerate}
\usepackage{enumitem}
\usepackage{hyperref}
\usepackage{xcolor}
\usepackage{mathtools}

\DeclarePairedDelimiter\floor{\lfloor}{\rfloor}
\long\def\symbolfootnote[#1]#2{\begingroup
\def\thefootnote{\fnsymbol{footnote}}\footnote[#1]{#2}\endgroup}
\theoremstyle{definition}

\newtheorem{thm}[subsubsection]{\bf Theorem}
\newtheorem{lem}[subsubsection]{\bf Lemma}
\newtheorem{cor}[subsubsection]{\bf Corollary}
\newtheorem{prop}[subsubsection]{\bf Proposition}
\newtheorem{rem}[subsubsection]{\bf Remark}
\newtheorem{defn}[subsubsection]{\bf Definition}

\newtheorem{ex}[subsubsection]{\bf Example}

\newcommand{\A}{{\mathbb A}}
\newcommand{\F}{{\mathbb F}}

\newcommand{\Q}{{\mathbb Q}}
\newcommand{\K}{{\mathbb K}}
\newcommand{\R}{{\mathbb R}}

\newcommand{\Z}{{\mathbb Z}}
\newcommand{\N}{{\mathbb N}}

\newcommand{\GL}{{\rm GL}}
\newcommand{\SL}{{\rm SL}}

\title[]{Automorphic tensor products and cuspidal cohomology of the ${\rm GL}_4$}
\date{\today}

\author{\bf Chandrasheel Bhagwat \ \& \ Sudipa Mondal}
\address{Indian Institute of Science Education and Research, Dr.\,Homi Bhabha Road, Pashan, Pune 411008,  INDIA.}
\email{cbhagwat@iiserpune.ac.in, \ sudipa.mondal123@gmail.com}


\begin{document}

\begin{abstract}
In this article, we establish an asymptotic lower bound estimate on the contribution of cuspidal automorphic representations of ${\rm GL}_4(\mathbb A_{\mathbb Q})$ to cuspidal cohomology of the ${\rm GL}_4$ which are obtained from automorphic tensor product of two automorphic representations of ${\rm GL}_2(\mathbb A_{\mathbb Q})$ of given weights and with varying level structure. In the end, we also prove that the symmetric cube of a representation of $\GL_2$ and the automorphic tensor product of  two representations of $\GL_2$ can not be equal (up to a twist by a character of $\GL_1$) to each other, under the suitable assumptions on the representations being cuspidal and cohomological. 
\end{abstract}
\smallskip

\maketitle
\section{Introduction}\label{introduction}

For a connected reductive algebraic group over a number field $\F$, one studies the notion of cohomological representations of $G(\A_\F)$ in terms of the sheaves on the associated ad\'elic symmetric spaces, defined by the highest weight representations of $G$.

It is interesting to ask how much of the cohomology of $G$ is captured by the automorphic representations of $G(\A_\F)$. In the case of $\GL_n$, such representations are often obtained by functorial transfer from products of lower rank $\GL_{n'}$ subgroups. In  \cite{chandrasheel-sudipa}, authors of this manuscript proved an estimate of the contribution of symmetric cube transfer from $\GL_2$ to the cohomology of $\GL_4$. \smallskip

The other possible functorial transfers for $\GL_4(\A_\Q)$  are automorphic tensor product from $\GL_2 \times \GL_2 $ to $\GL_4$ and Asai transfer from ${\GL_{2}} |_ {\K}$ for some imaginary quadratic extension $\K$ over $\Q$. In this manuscript, we prove the following estimate on the contribution of automorphic tensor products from $\GL_2 \times \GL_2 $ to cohomology of the $\GL_4$. 

 Let $k_1 > k_2  \geq 2$ be two positive integers. Suppose they are  of different parity. Without loss of generality, assume $k_1$ to be odd and $k_2$ even. Define $ \mathcal T_{k_1, k_2}(N)$ to be the set of all cuspidal automorphic representations of $\GL_4 (\mathbb{A}_{\Q})$ corresponding to the level structure $K_f^4(N)$ that are obtained by automorphic tensor product of two cuspidal automorphic representations of $\GL_2(\mathbb{A}_{\Q})$ with highest weights $\left(\dfrac{k_1-1}{2},~\dfrac{1-k_1}{2}\right)$ and $\left(\dfrac{k_2}{2}-1,~1-\dfrac{k_2}{2}\right)$, respectively. (If $k_1,~ k_2$ are of same parity, we consider the half-integral twist of the automorphic tensor product of two representations $\pi_1, \pi_2$ of $\GL_2(\mathbb{A}_{\Q})$ with suitable highest weight depending on $k_1,k_2$ to be even or odd). See section \ref{statement of main theorem} for the notations.\smallskip

\begin{thm}\label{main theorem-intro}
	Let $ k_1 > k_2 \geq 2$ be two integers. Then for any integer $N>24$,
		$$ |\mathcal{T}_{k_1,~k_2}(p^{N})| \gg_{k_1,~k_2} ~p^{2\floor{N/4}} \text{ as the prime }  p \longrightarrow \infty.$$
\end{thm}

We briefly describe the structure of this manuscript. After setting up preliminaries about cusp forms in Sec.~\ref{preliminaries}, we discuss the automorphic tensor product representations in Sec~\ref{Automorphic Tensor product} including a calculation on the conductor of these representations. In Sec.~\ref{cohomology}, we discuss briefly about the cohomological representations of $\GL_4(\A_\Q)$ and examples of cohomological automorphic tensor product representations. In Sec.~\ref{statement of main theorem}, we give the proof of Thm.~\ref{main theorem-intro} using the construction and criteria for cuspidality of automorphic tensor product for $\GL_2 \times  \GL_2$ proved by D. Ramakrishnan in \cite{ramakrishnan} and \cite{ramakrishnan-cuspidality}.
In Sec.~\ref{overlap}, we give a proof of the fact that the symmetric cube of a representation of $\GL_2$ and the automorphic tensor product of  two representations of $\GL_2$ can not be equal (up to a twist by a character of $\GL_1$) to each other, under the suitable assumptions on the representations being cuspidal and cohomological. This result may be possibly known to the experts and may be unsurprising. However, we add the details of the calculations and proofs for the purpose of completion.\smallskip

\section{Preliminaries}\label{preliminaries}

\subsection{Dimension formulae for spaces of cusp forms}\label{Dimension formulae for spaces of cusp forms}
	
We discuss some well-known and relevant results in this section. Let $N$ be an integer. Consider the  congruence subgroup $\Gamma_1(N)$ of $\text{SL}_2(\mathbb{Z})$ defined by
	\begin{align*}
	\Gamma_1(N) & = \left\{A \in \text{SL}_2(\mathbb{Z}): A \equiv \begin{bmatrix}
	1 & * \\
	0 & 1 
	\end{bmatrix}(\text{mod } N) \right\}.
    \end{align*}
    
For an integer $k$, let $S_k(\Gamma_1(N))$ and $S_k^\text{new}(\Gamma_1(N))$ denote the spaces of cusp forms of weight $k$ and newforms of weight $k$, respectively for the congruence subgroup $\Gamma_1(N)$. The following well-known results will be important in the proof of the main theorem (see  \cite{stein} for the details and \cite{ambi} for the formulation of these results as given below).\smallskip
	
	\begin{lem}\label{dimension-estimate-cusp-forms}
	   An estimate for the dimension of the space of cusp forms:
	    $$ \frac{\rm{dim}_{\mathbb{C}}(S_k(\Gamma_1(N)))}{N^2}= \frac{k-1}{4\pi^2}+o(1) \text{ as } N \longrightarrow \infty.$$
	\end{lem}\smallskip
	
	\begin{lem}\label{dimension-estimate-new-forms} A dimension formula for newforms for $\Gamma_1(N)$:
	For $n \geq 2$ and a prime $p \geq 2$, we have
		$$  \frac{{\rm dim}_\mathbb{C}~(S_k^\text{new}(\Gamma_1(p^n)))}{p^{2n}} = \left( \frac{k-1}{4\pi^2} \right) \left(1+ \frac{1}{p^2}-\frac{2}{p^4}\right) + o(1) \quad \text{ as } ~p \longrightarrow \infty.$$
	\end{lem} \smallskip

 \begin{proof}
  Let $ M\in \mathbb{N}$. For a prime $p$, let $v_p(M)$ denote the highest power of $p$ that divides $M$. For $M \in \N$, let $a_M$ be the the number of prime factors of $M$ such that $v_p(M) = 2$. We define a function $\bar{\mu}$ on $\N$ by,
    \[
    \bar{\mu}(M)= \begin{cases}
    0 & \text{ if } v_p(M) \geq 3 \text{ for some prime } p \\
    (-2)^{a_{M}} & \text{ otherwise } \\
     \end{cases}.
    \]
   From (\cite[ Prop.~6.6]{stein}), we know that
    $${\rm dim}_\mathbb{C}(S_k^\text{new}(\Gamma_1(M)))=\sum_{d|M} \bar{\mu}(M/d) \cdot {\rm dim}_{\mathbb{C}}(S_k(\Gamma_1(d))).$$
    For $M=p^n, n\geq2$, there are at most three terms such that  $\bar{\mu}(M/d) \neq 0$ in the above sum. Adding these three terms gives
    \begin{eqnarray*}
    {\rm dim}_\mathbb{C}(S_k^\text{new}(\Gamma_1(M))) = {\rm dim}_\mathbb{C}(S_k(\Gamma_1(p^n)))+  {\rm dim}_\mathbb{C}(S_k(\Gamma_1(p^{n-1})))  -2 \text{ dim}_\mathbb{C}(S_k(\Gamma_1(p^{n-2}))).
    \end{eqnarray*}
    Using Lem.~\ref{dimension-estimate-cusp-forms} and fixing $n$ while varying $p$, we get
    
    \begin{eqnarray*}
    \frac{{\rm dim}_\mathbb{C}(S_k^\text{new}(\Gamma_1(M)))}{p^{2n}}= \frac{k-1}{4\pi^2} \left(1+ \frac{1}{p^2}-\frac{2}{p^4}\right) + o(1) \text{ as } p \longrightarrow \infty. 
    \end{eqnarray*}
    \end{proof} 
    
    \subsection{An upper bound for normalised eigenforms}\label{an-upper-bound-for-eigenforms} Let $C_k(N)$ denote the set of normalised cusp eigenforms of Hecke operators of $\Gamma_1(N)$ of weight $k$ obtained by automorphic induction of gr\"o{\ss}encharacters of all possible imaginary quadratic extensions of $\Q$. From the work of Ambi \cite{ambi}, we have an upper bound on the cardinality of this set as described below. \medskip

\noindent  Notation: 	Let $f$ and $g$ be two non-negative real valued functions on $\mathbb{N}$. We write $f \ll g$ (or equivalently $g \gg f$)  if there exists a constant $C$ and $n_0\in \mathbb{N}$ such that 
$$f(n) \leq Cg(n) \quad \forall ~ n\geq n_0.$$ 
	
We write $f \sim g$ if both $f \ll g$ and $g \ll f$ hold.\medskip

\begin{thm}\label{theorem by ambi}(\cite{ambi})
For $k,N \geq 1$ and $\epsilon \in (0,1),$ let $N': = \prod \limits_{p ~\text{prime}, ~p | N} p$. Then
    $$|C_k(N)| ~\ll_{k,\epsilon} N \cdot { N'}^{1+\epsilon} \text{ as } N \longrightarrow \infty$$
\end{thm}\smallskip

If $N=p^n ,n\in \mathbb{N} $, then more is true.
\begin{thm}\label{slight modification}
    Let $k\geq 2$ and $n\in \N$. We have, 
$$|C_k(p^n)| \ll_{\epsilon} p^{n+\frac{1}{2}+\epsilon} \text{ as } \epsilon \longrightarrow 0,~ p \longrightarrow \infty.$$  
\end{thm}

The following auxiliary results are discussed in \cite{ambi} and \cite{chandrasheel-sudipa} respectively. We mention them here for a reference.\smallskip

\begin{lem}\label{auxiliary} Suppose $p$ is an odd prime. 
\begin{enumerate}
\item If $p \equiv 1 \mod 4$, then 
$$|C_k(p^n)|=0 \quad \forall ~\text{ } k\geq 2 \text{ and } \forall \text{ } n\geq1.$$ 

\item If  $p \equiv 3 \mod 4$, $\mathbb{E}=\Q(\sqrt{-p})$ and $\mathfrak f \subseteq \mathcal{O}_{\mathbb{E}}$ such that  $N_{\mathbb{E}/\mathbb{F}}(\mathfrak f) \cdot \allowbreak |D_{\mathbb{E}/\mathbb{Q}} ~|| ~p^n$ for some $n\geq 1$, then $\mathfrak f~ |~ p^{m} \mathcal{O}_{\mathbb{E}}$ where $m= \floor{\dfrac{n}{2}}$.
\end{enumerate}
\end{lem}\smallskip

We now prove Thm. \ref{slight modification}. We need to only consider the case $p \equiv 3 \mod 4$ case in view of Lem. \ref{auxiliary}.

\begin{proof} 

Let, for an ideal $ \mathfrak{J} \subset \mathcal{O}_{\mathbb{E}}$, $h_{\mathbb{E}}^0(\mathfrak{J})$  and $h_{\mathbb{E}}'(\mathfrak{J})$ denote the number of Hecke characters with conductor exactly $\mathfrak{J}$ and the number of characters of the narrow ray class group modulo $\mathfrak{J}$ respectively. Hence $\sum \limits_{\mathfrak f~|~\mathfrak{J}} h_{\mathbb{E}}^0(\mathfrak f)=h_{\mathbb{E}}'(\mathfrak{J})$. Using (Theorem 12.5, \cite{iwaniec})  we have,
$$|C_k(p^n)| \leq \sum \limits_{\substack{\mathbb{E}=\mathbb{Q}(\sqrt{-p}), \\  N_{\mathbb{E}/\mathbb{Q}}(\mathfrak f) \cdot|D_{\mathbb{E}/\mathbb{Q}}~||~p^n}} h_{\mathbb{E}}^0(\mathfrak f) .$$

By part (2) of Lem.~\ref{auxiliary}, 
$$|C_k(p^n)|\leq \sum \limits_{\substack{\mathfrak f~|~p^m\mathcal{O}_{\mathbb{E}} \\ \mathbb{E}=\mathbb{Q}(\sqrt{-p}) }} h_{\mathbb{E}}^0(\mathfrak f)= h'_{\mathbb{E}}(p^m\mathcal{O}_{\mathbb{E}}) \leq  h_{\mathbb{E}}\Phi_{\mathbb{E}}(p^m\mathcal{O}_{\mathbb{E}}),$$ 
where $h_{\mathbb{E}}=$ the class number of $\mathbb{E}$ and $\Phi_{\mathbb{E}}(\mathcal{I})= N_{\mathbb{E}/\mathbb{Q}}(\mathcal{I})\prod \limits_{\substack{\mathfrak{p}|\mathcal{I}, \\ 
\mathfrak{p} \text{ prime }}} (1-N_{\mathbb{E}/\mathbb{Q}}(\mathfrak{p})^{-1})$ with $\Phi_{\mathbb{E}}(\mathcal{O}_{\mathbb{E}})=1$. \smallskip

The last inequality follows from \cite[Theorem 3.25(i)]{narkiewicz}. Hence,
$$ |C_k(p^n)| \leq h_{\mathbb{E}} \cdot p^{2m}.$$

We use the following asymptotic formula for the class number as in \cite[Corollary 1 of Theorem 8.14]{narkiewicz}.
\[
\begin{split}
& \quad {\log} ~h_{\Q({\sqrt{-d}})} = \left( \frac{1}{2}+o(1) \right) \log |d|  \quad\text{ as } ~d \longrightarrow \infty \\
\implies & \quad  h_{\Q({\sqrt{-d}})} \ll_{\epsilon} d^{{\frac{1}{2}}+\epsilon} \quad \text{ as } \epsilon \longrightarrow 0, ~d \longrightarrow \infty\\
\implies & \quad |C_k(p^n)| \ll_{\epsilon} p^{2m+\frac{1}{2}+1} \text{ as } \epsilon \longrightarrow 0, ~p \longrightarrow \infty \\
\implies & \quad |C_k(p^n)| \ll_{\epsilon} p^{n+\frac{1}{2}+\epsilon} \text{ as } \epsilon \longrightarrow 0, ~p \longrightarrow \infty.
\end{split}
\]
\end{proof}

 \subsection{The Level Structure :}  \label{level-structure}
 Consider the algebraic group ${\rm GL}_m$. We very briefly recall the definition of a level structure for ${\rm GL}_m(\mathbb{A}_\Q)$. For a finite prime $p$ of $\Q$, let $\Z_p$ be the ring of $p$-adic integers. For each integer $n \geq 0$, define
 $$ K_p^m(n):= \{ x=(x_{i,j})_{m\times m} \in {\rm GL}_m(\mathbb{Z}_p) : x_{m,k}\in p^n\mathbb{Z}_p, ~\text{ } 1\leq k<m, ~x_{m,m}-1 \in p^n\Z_p \}.$$
 
 Let $\mathbb{A}_f$ denote the finite part of ad{\`e}les over $\Q$. Let $N$ be a positive integer with prime factorisation $N =  \prod \limits_{1 \leq i \leq r, p_i | N}p_i^{n_i}$. Define a compact open subgroup $K_f^m(N) = \prod \limits_p K_p$ of ${\rm GL}_m(\mathbb{A}_f)$ where 
 $$K_p =
 \begin{cases} 
 & K_{p_i}^m{(n_i)} \quad \text{if} ~ p \mid N~\text{i.e.},~ p = p_i  \\
 & {\rm GL}_m(\mathbb{Z}_p) \quad \text{if}~ p \nmid N
 \end{cases}.
 $$
 
 For each $N \ge 1$, the compact open subgroup $K_f^m(N) \subseteq {\rm GL}_m (\mathbb A_f)$ is called the level structure corresponding to $N$.
\smallskip

\section{Automorphic Tensor product}\label{Automorphic Tensor product}
	
\subsection{Definition} We recall the definition of automorphic tensor product transfer from  $\rm{GL}_2 \times \rm{GL}_2 $ to $\rm{GL}_4$ over a number field $\F$. Let $\pi_1 = \bigotimes \limits_{v} \pi_{1,v}$ and $\pi_2 = \bigotimes \limits_{v} \pi_{2,v}$ be two cuspidal automorphic representations of $\rm{GL}_2(\A_\F)$. By local Langlands correspondence for each place $v$, the local  constituents $\pi_{1,v}$ and $\pi_{2,v}$ correspond to two dimensional representations of the local Weil-Deligne group. 

D. Ramakrishnan in \cite[Thm. M]{ramakrishnan} proved the existence of an automorphic representation (automorphic tensor product) $\pi_1 \boxtimes \pi_2$ of  $\rm{GL}_4(\A_\F)$ that satisfies:
\begin{itemize}
\item Equality of local $L$-functions: $L(s,(\pi_1 \boxtimes \pi_2)_v) = L(s, \pi_{1,v} \times \pi_{2,v})$  for all places $v$, \smallskip
\item Equality of local $\epsilon$-factors: $ \varepsilon(s,(\pi_1 \boxtimes \pi_2)_v) = \varepsilon(s, \pi_{1,v }\times \pi_{2,v})$ for all finite places $v$.
\end{itemize}
\smallskip

\subsection{The cuspidality criterion for automorphic tensor product}\label{cuspidality}

We describe the cuspidality criterion for the automorphic tensor product (proved by D. Ramakrishnan in  \cite{ramakrishnan-cuspidality}) (we need it only for the case of base field = $\Q$).\smallskip

\begin{thm}\label{cuspidality criterion}
	Let $\pi_1, ~\pi_2$ be two cuspidal automorphic representations of $\GL_2(\mathbb{A}_{\Q})$. Then $\pi_1 \boxtimes\pi_2$ is an isobaric automorphic representation of $\GL_4(\mathbb{A}_{\Q})$. Further more $\pi_1 \boxtimes\pi_2$ is cuspidal under the following conditions.\smallskip
	
\begin{enumerate}
	
	\item  If exactly one of them is dihedral (i.e. automorphically induced from imaginary quadratic extensions of $\Q$), then $\pi_1 \boxtimes\pi_2$ is cuspidal.\smallskip
	
	\item If both of them are non-dihedral, then $\pi_1 \boxtimes\pi_2$ is cuspidal if and only if $\pi_2$ is not of the form $\pi_1 \otimes \chi$ for any id\'ele class character $\chi$ of $\Q$.\smallskip
	
	\item Suppose $\pi_1,~\pi_2$ are both dihedral, say
	$\pi_i = I^{\Q}_{\K_{i}}$ for $i = 1, ~ 2$, where $\K_i$'s are quadratic extensions of $\Q$ and $\mu_1$  (resp., $\mu_2$) is an id\'ele class character of $\K_1$ (resp., $\mathbb{K}_2$). Then $\pi_1 \boxtimes \pi_2$ is cuspidal if and only if $\K_1 \neq \K_2$.\smallskip
	\end{enumerate}
\end{thm}

\smallskip

\subsection{Conductors for automorphic tensor products}\label{conductors}
 
Conductor of an automorphic  representation: We briefly recall the notion of the conductor of an automorphic representation here.\smallskip

Let $(\rho,\mathcal H)=(\bigotimes \limits_{p\leq \infty}\rho_p, \bigotimes \limits_{p\leq \infty}\mathcal H_p)$ be an irreducible automorphic representation of ${\rm GL}_m(\mathbb{A}_\Q)$. 
\smallskip

\begin{defn} For each finite prime $p$, the conductor of $\rho_p$ is defined to be  the smallest  integer $c(\rho_p)\geq 0$ such that the space $\mathcal{H}_p^{K_p^m({{c(\rho_p)}})}$ consisting of all $K_p^m({{c(\rho_p)}})$-fixed vectors of $\mathcal H_p$ is non-zero. Thus $c(\rho_p) = 0$ for almost all primes. The conductor of $\rho$ is defined as 
	$$N_\rho= \prod_{p<\infty} p^{c(\rho_p)} .$$ 
\end{defn}\smallskip
	
\noindent{\bf Conductor of a local character:} Let $\mathcal{K}$ be a non-archimedean local field of characteristic zero. Let $\mathcal{O}_{\mathcal{K}}$, $\mathfrak{p}$, $\kappa$ be its ring of integers, the maximal ideal and the residue field respectively. \smallskip
		
\begin{defn}	
The conductor $c(\zeta)$ of a character $\zeta$ of $\mathcal{K}^\times$ is the smallest positive integer $n$ such that ${\zeta}_{|_{1+\mathfrak{p}^n}} = 1 $. 
\end{defn}
\smallskip
	
\noindent{\bf A relation between level and conductor:} For a supercuspidal representation $\pi$ of $\GL_n(\Q_p)$, if $c(\pi)$ is the conductor of $\pi$, we have $c(\pi)= n \left(1+\dfrac{m}{e} \right)$, where $m$ is the level of $\pi$ and $e$ is the ramification index (see \cite[p.149]{anandavardhanan}). We will consider only the case when $n=2$. Note that, $c(\pi) \geq 2$.\medskip
	
Since, $1\leq e\leq2$ , we have 
\begin{equation}\label{conductor-inequality}
\begin{split} & \quad  \dfrac{m}{2}\leq \dfrac{m}{e}\leq m \\
\implies & \quad \frac{m}{2} \leq \dfrac{c(\pi)}{2}-1 \leq m \\
\implies& \quad m+2 \leq c(\pi) \leq 2m+2.
\end{split}
\end{equation}	
	
\begin{lem}\label{conductor relation}
Let $p$ be an odd prime and $\pi_1, ~\pi_2$ be two supercuspidal representations of $\rm{GL}_2(\Q_p)$. Let $c(\pi_1),~c(\pi_2) \text{ and } c(\pi_1 \times \pi_2)$ be the conductors of $\pi_1, ~\pi_2 \text{ and } \pi_1 \times \pi_2$ respectively. Then we have,
$$ 2 \leq c(\pi_1 \times \pi_2) \leq 4c(\pi_1)+4c(\pi_2)-12.$$ 
\end{lem}
	
\begin{proof}
Using  \cite[Thm. 6.5 ]{bushnell} and the calculation before the Lem.~\ref{conductor relation}, we get the desired inequality.\smallskip
	
Take $e= {\rm l.c.m.}(e_1,e_2)$, and $m_1$, $m_2$ be the levels of $\pi_1$ and $\pi_2$, respectively. Let $m$ be defined by $$m = \max \left\{ \frac{e m_1}{e_1},\frac{em_2}{e_2} \right\}.$$
	
As in the notation in \cite{bushnell}, we have
$$n_1=n_2=2, \quad 1\leq e_1,e_2 \leq 2,$$
$$1\leq d\leq 2, \quad  0\leq c(\beta),c(\gamma), \quad l \leq m.$$
	
Let $\check{\pi_1}$ denote the contragredient representation of $\pi$. Using \cite[Thm. 6.5 ]{bushnell}, the inequality \eqref{conductor-inequality}, and the fact $c(\pi)=c(\check{\pi})$  (see \cite[Prop., p.23]{bushnell1}), we conclude that
$$2 \leq c(\pi_1 \times \pi_2) \leq 4c(\pi_1)+4c(\pi_2)-12.$$
\end{proof}

\section{Cohomology of the $\GL_m$}\label{cohomology}

\subsection{Cohomological representations of ${\GL}_m(\A_\Q)$ }\label{Coh-rep}

Let $\pi$ be a cuspidal automorphic representation of ${\GL}_m(\A_\Q)$ with finite and infinite parts denoted by $\pi_f $ and $\pi_\infty$ respectively. Let $\mu$ be a dominant integral weight corresponding to the standard Borel subgroup of ${\GL}_m$ and $\mathcal M_{\mu,\mathbb{C}}$ be the underlying vector space of the finite dimensional irreducible highest weight representation corresponding to $\mu$. For a given level structure $K_f^m \subset  {\rm GL}_m(\mathbb{A}_f)$, let $\pi^{K^m_f}_f$ denote the space of $K_f^m$-fixed vectors of $\pi_f$.\smallskip

 Define  $ {\rm Coh}({\GL}_m, \mu, K_f^m)$ = the set of all automorphic representations $\pi$  such that its relative Lie algebra cohomology with coefficients in $\mathcal M_{\mu,\mathbb{C}}$ is non zero in some degree i.e.,
	$$ H^*(\mathfrak{gl}_m, \mathbb{R}_+^\times \cdot {\rm SO}_m(\mathbb{R}),\pi_\infty \bigotimes \mathcal M_{\mu,\mathbb{C}})\bigotimes \pi_f^{K_f^m} \neq 0 .$$

We say that $\pi$ is cohomological if $\pi \in {\rm Coh}({\GL}_m, \mu, K_f^m)$ for some highest weight $\mu$ and level structure $K_f^m$. Now we state a necessary condition for $\pi$ to be cohomological. Let $\mathcal{L}_0^+(\GL_m)$ be the set of all pairs $( w, \ell)$ with $\ell = (\ell_1,\ell_2, \dots , \ell_m)\in \Z^m$ such that $\ell_1 > \ell_2 > \dots > \ell_m >0$  and $\ell_i=-\ell_{m-i+1}$ and $w\in \Z$ such that for all $i$, 
\[
    w+\ell_i \equiv \begin{cases}
    1 \text{ if } m \text{ is even }\\
    0 \text{ if } m \text{ is odd }
    \end{cases}.
\]
For $(w,\ell) \in \mathcal{L}_0^+(\GL_m)$, let $D_{\ell_i} $ be the discrete series representation of $\rm{GL}_2(\R)$ of lowest weight $\ell_i+1$, $1 \leq i \leq m$. Define,	
\[
J(w,\ell)= \begin{cases}
{\rm Ind}_{P_{2,2,\dots,2}}^{{\rm GL}_{m}}
\left(
(D_{\ell_1}\otimes|.|^{w/2}_{\R})\otimes \dots \otimes(D_{\ell_{m/2}} \otimes |.|^{w/2}_{\R})
\right)
 & \text{ if } m \text{ is even} \\
{\rm Ind}_{P_{2,2,\dots,2,1}}^{{\rm GL}_m}
\left(
(D_{\ell_1}\otimes|.|^{w/2}_{\R})\otimes \dots \otimes(D_{\ell_{(m-1)/2}}\otimes|.|^{w/2}_{\R}) \otimes |.|^{w/2}_{\R}
\right)
 & \text{ if } m \text{ is odd}
\end{cases}.
\]

If $\pi$ is cuspidal cohomological representation of $\GL_m(\mathbb{A}_{\Q})$, then the representation $\pi_{\infty}$ must be equal to $J(w,\ell)$ for some $(w,\ell) \in \mathcal{L}_0^+(\GL_m)$. See (\cite{raghuram-shahidi}, section 5) for more details.\smallskip

\subsection{Examples}\label{Examples} We are interested in studying the contribution to the $\GL_4$-cohomology of the representations that are obtained via automorphic tensor products of two cusp forms on $\GL_2$. For that, consider the following examples. \smallskip

\begin{ex} (see \cite{raghuram-shahidi}) Let $k_1 > k_2 \geq 2$ be positive integers. Let  $\phi_1 \text{ and } \phi_2$ be two cusp forms of weight $k_1,k_2$ respectively and $\pi_{\phi_{1}}$ and $\pi_{\phi_{2}}$ be the corresponding automorphic cuspidal representations of $\GL_2(\A_\Q)$. Let $\pi: = \pi_{\phi_{1}} \boxtimes \pi_{\phi_{2}}$ denote their automorphic tensor product.  $\pi_\infty$ = the representation at infinity as a representation of $\GL_4(\R)$, can be computed using a well-known recipe for Langlands parameters as
 $$ \pi_{\infty}= {\rm{Ind}}_{P_{2,2}}^{\GL_4} (D_{k_1+k_2-2} \otimes D_{k_1-k_2} ).$$
 
 We have to consider two cases based on the parity of $k_1 +k _2$.\smallskip
 
 {\bf Case (1): } Suppose $k_1 + k_2 $ is odd.  In this case we see that  $\pi_{\infty} = J(0,\ell), $ where: 
 $$ \ell= (k_1 + k_2-2,~ k_1-k_2,~ k_2-k_1, ~-k_1-k_2+2).$$

   Hence $\pi$ is cohomological with respect to the dominant integral weight
    $$\mu_{\pi} = \nu_{k_1,k_2}= \left( \frac{k_1+k_2-5}{2}, ~\frac{k_1-k_2-1}{2}, ~\frac{k_2-k_1+1}{2},~ \frac{-k_1-k_2+5}{2}\right).$$\smallskip
   
    {\bf Case (2): } Suppose $k_1 + k_2 $ is even.  In this case we see that $\pi$ is not cohomological but its half-integral Tate-twist $\pi \otimes | \cdot |^{1/2}$ is cohomological. Moreover,
    $(\pi \otimes | \cdot |^{1/2})_{\infty} = J(1,\ell), $ where 
 $$ \ell= (k_1 + k_2-2,~ k_1-k_2,~ k_2-k_1, ~-k_1-k_2+2).$$
 
   Hence $\pi$ is cohomological with respect to the dominant integral weight $$\mu_{\pi} = \left( \frac{k_1+k_2 - 4}{2}, ~\frac{k_1-k_2}{2}, ~\frac{k_2-k_1+2}{2},~ \frac{-(k_1+k_2)+ 6}{2} \right).$$
   
   \end{ex}\smallskip

\begin{ex} The following example from (\cite{raghuram-shahidi}, section 5) shows that we must consider two cusp forms of different weight for their automorphic tensor product to be cohomological (up to a half-integral Tate-twist).\smallskip

Let $\phi_1,\phi_2$ are two cusp forms of weight $k \in \Z, k$ even such that $\pi=\pi_{\phi_1}\boxtimes \pi_{\phi_2}$ is cuspidal. Then the similar computation involving the Langlands parameters gives that the infinite part of $\pi$ is given by 
\[ \pi_{\infty}= {\rm Ind}_{P_{2,1,1}}^{\GL_4} (D_{2(k-1)} \otimes sgn \otimes 1).
\]

From the shape of $\pi_{\infty}$, it follows that neither $\pi_\infty$ nor $(\pi \otimes | \cdot |^{1/2})_{\infty}$ are of the type $J(w, \ell)$ for any $ (w,\ell) \in \mathcal{L}_0^+(\GL_m)$ and hence neither $\pi$ nor  $\pi \otimes | \cdot |^{1/2}$ are  cohomological.
\end{ex}
\smallskip

\section{An estimate of the contribution of automorphic tensor products to cohomology of the $\GL_4$}\label{statement of main theorem}

\subsection{Notations}\label{Notations}

We develop a few required notations for the statement of the main theorem in this subsection.\medskip

\begin{itemize}
\item Let $k_1 > k_2  \geq 2$ be two positive integers. Suppose they are  of different parity. Without loss of generality, assume that $k_1$ is odd and $k_2$ is even. Define $ \mathcal T_{k_1, k_2}(N)$ to be the set of all cuspidal automorphic representations of $\GL_4 (\mathbb{A}_{\Q})$ corresponding to the level structure $K_f^4(N)$ that are obtained by automorphic tensor product of two cuspidal automorphic representations of $\GL_2(\mathbb{A}_{\Q})$ with highest weights $\left(\dfrac{k_1-1}{2},~\dfrac{1-k_1}{2}\right)$ and $\left(\dfrac{k_2}{2}-1,~1-\dfrac{k_2}{2}\right)$, respectively. (If $k_1,~ k_2$ are of same parity, we consider the half-integral twist of the automorphic tensor product of two representations $\pi_1, \pi_2$ of $\GL_2(\mathbb{A}_{\Q})$ with suitable highest weight depending on $k_1,k_2$ are even or odd).
\smallskip

\item Let $C^0_{k_1,~k_2}(N_1,N_2)$ denote the set of ordered pairs $(f_1,f_2)$ of all normalized cusp eigenforms $f_1,f_2$ of Hecke operators of $\Gamma_1(N_1)$ of weight $k_1$ and $\Gamma_1(N_2)$ of weight $k_2$ respectively such that both $f_1$ and $f_2$ are dihedral, and are obtained by automorphic induction from Hecke characters of (same) imaginary quadratic extension of $\Q$.
\smallskip

Note that, $C^0_{k_1,~k_2}(N_1,N_2) \subseteq C_{k_1}(N_1) \times C_{k_2}(N_2)$ (see the Sec.~\ref{an-upper-bound-for-eigenforms} for notations). Hence, if $N_1=p^{n_1} \text{ and } N_2=p^{n_2}$, using Thm.~\ref{slight modification} we have,
$$ |C^0_{k_1,~k_2}(p^{n_1},p^{n_2})| \ll_{\epsilon} p^{n_1+n_2+2+2\epsilon} \text{ as } p \longrightarrow \infty.$$

 \item Let $C'_{k_1,~k_2}(N_1,N_2)$ denote the set of ordered pairs $(g_1,g_2)$ of normalized cusp eigenforms $g_1, ~g_2$ of Hecke operators of $\Gamma_1(N_1)$ of weight $k_1$ and $\Gamma_1(N_2)$ of weight $k_2$ respectively such that $g_1, ~g_2$ are both non-dihedral and one of them is the twist of the other by some id\'ele class character of $\Q$.\smallskip

Note that under the assumption that $k_1 > k_2,$ the set $C'_{k_1,~k_2}(N_1,N_2)$ is indeed an empty set.\smallskip
\end{itemize}

\subsection{The main theorem}

\begin{thm}\label{main theorem}
	Let $ k_1 > k_2 \geq 2$ be two integers. Then for any integer $N>24$,
		$$ |\mathcal{T}_{k_1,~k_2}(p^{N})| \gg_{k_1,~k_2} ~p^{2\floor{N/4}} \text{ as the prime }  p \longrightarrow \infty.$$
\end{thm}
\begin{proof} We first consider the {\it different parity} case where $k_1 + k_2$ is odd, $k_1$ is odd and $k_2$ is even. Let $p$ be an odd prime. Suppose $\pi_1 ,\pi_2$ are cuspidal automorphic representations of $\GL_2(\mathbb{A}_{\Q})$ corresponding to the level structure $K_f^2(p^{n_1})$ and $K_f^2(p^{n_2})$, respectively, with $n_1, n_2\geq1$ and of weights $\left(\dfrac{k_1-1}{2},\dfrac{1-k_1}{2}\right)$ and $\left(\dfrac{k_2}{2}-1,~1-\dfrac{k_2}{2}\right)$, respectively.\smallskip

Using Thm. \ref{cuspidality criterion}, we conclude that exactly one of the following holds:\smallskip
\begin{itemize}
\item $(\pi_1 ,\pi_2) \in C'_{k_1,~k_2}(p^{n_1},p^{n_2})$\smallskip

\item $(\pi_1 ,\pi_2) \in C^0_{k_1,~k_2}(p^{n_1},p^{n_2})$ 
\smallskip

\item $\pi_1 \boxtimes \pi_2 \in \bigcup \limits_{n\geq1} \mathcal{T}_{k_1,~k_2}(p^n)$.
\smallskip
\end{itemize}

 Let $\pi_1=\bigotimes_{q\leq \infty} \pi_{1,q}$ and $\pi_2=\bigotimes_{q\leq \infty} \pi_{2,q}$. We know that $\pi_{1,p}$ (resp $\pi_{2,p}$) corresponds to a unique newform in $S_{k_1}^{\rm{new}}(\Gamma_1(p^{s_1}))$ (resp. $S_{k_2}^{\rm{new}}(\Gamma_1(p^{s_2}))$) with $1 \leq s_1 \leq n_1$ (resp $1 \leq s_2 \leq n_2$). Note that, $c(\pi_{1,p})=s_1$ and $c(\pi_{2,p})=s_2$.
 
 We want to count the contributions of supercuspidal representations $\pi_{i,p}, i=1,2$ to $\bigcup \limits_{n\geq1} \mathcal{T}_{k_1,~k_2}(p^n)$. Using Lem.~ \ref{conductor relation}, we can conclude that two newforms in $S_{k_1}^{\rm{new}}(\Gamma_1(p^{s_1}))$ and $S_{k_2}^{\rm{new}}(\Gamma_1(p^{s_2}))$ that are obtained from the supercuspidal representations $\pi_{1,p}$ and $\pi_{2,p}$ respectively, can contribute to $\mathcal{T}_{k_1,~k_2}(p^n)$ only if $$2 \leq n \leq 4c(\pi_{1,p})+4c(\pi_{2,p})-12.$$
 
 Since $1 \leq s_i \leq n_i, i=1,2 $, we have, 
 $$ 2 \leq n \leq 4n_1+4n_2-12.$$
 Note that, $n_i \geq 2$ since $s_i\geq 2$ as $\pi_{i,p}$ is supercuspidal, $i=1,2$. Hence, we conclude that, automorphic tensor product of any pair of supercuspidal representations in the set  
 
 $\bigoplus \limits_{1\leq i\leq n_1}S_{k_1}^\text{new}(\Gamma_1(p^i)) \times \bigoplus \limits_{1\leq i\leq n_2}S_{k_2}^\text{new}(\Gamma_1(p^i)) ~ \setminus~ \{ C^0_{k_1,~k_2}(p^{n_1},p^{n_2}) \cup C'_{k_1,~k_2}(p^{n_1},p^{n_2}) \} $
  is also in the set $\mathcal{T}_{k_1,~k_2}(p^{4n_1+4n_2-12})$. \smallskip
 
 Since we have the following bound
 $$ \sum \limits_{1\leq i\leq n_1}{\rm dim}_{\mathbb{C}} S_{k}^\text{new}(\Gamma_1(p^i)) \gg_k p^{2n_1}$$
  and we also have
 $$ | C^0_{k_1,~k_2}(p^{n_1},p^{n_2}) \cup C'_{k_1,~k_2}(p^{n_1},p^{n_2}) | \ll_{k_1,~k_2} p^{n_1+n_2+4}, $$ 
  which is negligible compared to $p^{2n_1+2n_2}$ for $p \longrightarrow \infty$ and $n_1+n_2 > 4$, we conclude that
 $$ p^{2n_1+2n_2} \ll_{k_1,~k_2} |\mathcal{T}_{k_1,~k_2}(p^{4n_1+4n_2-12})|   \text{ as }  p \longrightarrow \infty \text{ and } n_1+n_2 > 4.$$
In particular, taking $N = 4(n_1+n_2)$ with $N>24$, we have,	
  $$ |\mathcal{T}_{k_1,~k_2}(p^{N})| \gg_{k_1,~k_2} p^{2 \floor{N/4}} \text{ as }  p \longrightarrow \infty.$$
 
 The {\it same parity} case where $k_1 + k_2$ is even can be handled in an exactly same manner by considering $(\pi_1 \boxtimes \pi_2) \otimes | \cdot |^{1/2}$ instead of $\pi_1 \boxtimes \pi_2$. 
\end{proof} 
\smallskip

 Using Lem. \ref{dimension-estimate-cusp-forms} we have an upper bound as stated below. \smallskip
 
\begin{cor}
Let $N \in \mathbb{N}$.  Then, 
$$|\mathcal{T}_{k_1,~k_2}(p^{N})| \ll_{k_1,~k_2} ~p^{2N} \text{ as } p \longrightarrow \infty.$$
\end{cor}

\section{Possible Overlap between symmetric cube and automorphic tensor product representations of $\GL_4$}\label{overlap}

It is natural to investigate the overlap between the two types of contributions to the cohomology of $\GL_4$, the symmetric cube of a representation of $\GL_2$ and the automorphic tensor product of two representations of $\GL_2$.
\smallskip

\subsection{Comparison of representations at infinity} We first describe the representations at infinity. Now suppose $k_1, k_2, k_3 $ are integers $\geq 2$, $k_1 > k_2$ and they are weights of cusp forms corresponding to three automorphic representations $\pi_1, \pi_2, \pi_3$ of $\GL_2$, respectively. \smallskip

From the discussion in Sec. \ref{Examples}, we conclude that for a half-integer $s$,

 \begin{equation}\label{auto-tensor-infty}
((\pi_1 \boxtimes \pi_2)  \otimes | \cdot |^s )_{\infty} =
\begin{cases}
J(0,\ell)  \otimes | \cdot |_\infty^s \quad & \text{if}~k_1 + k_2 ~\text{is odd} \\
J(1,\ell) \otimes | \cdot |_{\infty}^{s-1/2} & \text{if}~k_1 + k_2 ~\text{is even}
\end{cases}
 \end{equation}
 where in both cases,
 $ \ell= (k_1 + k_2-2,~ k_1-k_2,~ k_2-k_1, ~-k_1-k_2+2).$ \medskip

On the other hand, $({\rm Sym^3}(\pi_3) \otimes | \cdot |^s )_{\infty}$ is given by \cite[Thm 5.5]{raghuram-shahidi}

 \begin{equation}\label{sym-3-infty}
(  {\rm Sym^3}(\pi_3)  \otimes | \cdot |^s )_{\infty}= 
 \begin{cases}
 J(1, \ell') \otimes | \cdot |_{\infty}^{s-1/2}\quad & \text{if}~k_3 ~\text{is odd} \\
J(0,\ell')  \otimes | \cdot |_{\infty}^{s} & \text{if}~k_3~\text{is even}
  \end{cases}
 \end{equation}
 where in both cases,
 $ \ell'= (3k_3 -3, ~k_3 -1, ~1- k_3, ~3- 3k_3).$\smallskip

Refer to Sec. \ref{Coh-rep} for the discussion of cohomological representations $J(w, \ell)$. It follows that for a half-integer $s$ (i.e. $2s \in \Z$), the twisted representation $J(w, \ell)  \otimes | \cdot |^s_{\infty}$ is cohomological iff $s \in \Z$ (i.e. $2s$ is even) and in that case,
\begin{equation} \label{Tate-twist-coh} J(w, \ell)  \otimes | \cdot |^s_{\infty} = J(w+2s, \ell)
\end{equation}

We analyze all possible scenarios as follows. \smallskip

\begin{enumerate}[label=(\alph*)]
\item Suppose $k_1+k_2$ and $k_3$ are even. Let $s$ be a half-integer (i.e. $2s \in \Z$). Then from \eqref{auto-tensor-infty} and  \eqref{Tate-twist-coh}, we conclude that $(\pi_1 \boxtimes \pi_2 )\otimes | \cdot |^{s} $ is cohomoloigcal if and only if $2s$ is odd, and in that case it equals $J(2s, \ell)$.
Since ${\rm Sym^3}(\pi_3)_{\infty} = J(0, \ell')$, we conclude that ${\rm Sym^3}(\pi_3) \otimes | \cdot |^{s} \neq 
\pi_1 \boxtimes \pi_2$ for any half integer $s$.\smallskip

\item Suppose  $k_1+k_2$ and $k_3$ are odd. Let $s$ be a half-integer (i.e. $2s \in \Z$). Then from \eqref{sym-3-infty} and \eqref{Tate-twist-coh} we conclude that
   ${\rm Sym^3}(\pi_3) \otimes | \cdot |^{s}$ is cohomoloigcal if and only if $2s$ is odd, and in that case it equals $J(2s, \ell')$. Since
  $
(\pi_1 \boxtimes \pi_2)_{\infty} = J(0, \ell)
$, we conclude that ${\rm Sym^3}(\pi_3) \otimes | \cdot |^{s} \neq 
\pi_1 \boxtimes \pi_2$ for any half integer $s$.\smallskip

\item Suppose $k_1 + k_2$ and $k_3$ are of different parity and $k_1$ is even. Thus  $k_2 , k_3$ have different parity. From  \eqref{auto-tensor-infty} and \eqref{sym-3-infty}, we conclude that: \smallskip

\noindent $((\pi_1 \boxtimes \pi_2)  \otimes | \cdot |^s )_{\infty} = {\rm Sym^3}(\pi_3)_{\infty}$ for some half integer $s$ (i.e. $2s \in \Z$) if and only if $s = 0$ and $\ell = \ell'$. \smallskip

If $\ell = \ell'$, then we have $$\frac{k_1 - 1}{2} = k_2 = k_3$$ 
This is impossible since we assumed that $k_1$ is even and $k_2 , k_3$ are of different parity. So we conclude that ${\rm Sym^3}(\pi_3) \otimes | \cdot |^{s} \neq 
\pi_1 \boxtimes \pi_2$ for any half integer $s$.
\smallskip

\item Suppose $k_1 + k_2$ and $k_3$ are of different parity and $k_1$ is odd. Thus $k_2 , k_3$ are of same parity. Similar to the previous case, we have the following necessary conditions for ${\rm Sym^3}(\pi_3) \otimes | \cdot |^{s} =
\pi_1 \boxtimes \pi_2$ for some half integer $s$  (i.e. $2s \in \Z$):\smallskip

\begin{itemize}
\item $\ell = \ell'$, i.e. $\dfrac{k_1 - 1}{2} = k_2 = k_3$\smallskip
\item $s = 0$
\end{itemize}\smallskip

This case is inconclusive and we can not rule out the possibility of an overlap using the method of comparing the representations at infinity.

\end{enumerate}
\smallskip

\subsection{Proof of `lack of overlap'}

 Now we prove that the overlap is in fact impossible up to any twist by a Hecke character of $\GL_1(\A_\Q)$ under suitable conditions on the representations being cuspidal and cohomological.\smallskip

\begin{prop}\label{no-overlap}
Let $\pi_1, \pi_2, \pi_3$ be three cuspidal automorphic representations of ${\GL}_2{(\mathbb{A}_{\Q})}$ such that they are given by cusp forms of weights $k_1, k_2, k_3 \geq 2$ and $\pi_1 \boxtimes \pi_2$ and ${{\rm Sym}}^3{(\pi_3)}$ are cuspidal automorphic representations of ${\GL}_4{(\mathbb{A}_{\Q})}$. \smallskip

Then there is no Hecke character of $\GL_1(\A_\Q)$ such that~
$ \pi_1 \boxtimes \pi_2 = {\rm Sym}^3{(\pi_3)} \otimes \chi$.
\end{prop}

\begin{proof}
Suppose $\pi_1 \boxtimes \pi_2 = {\rm Sym}^3{(\pi_3)} \otimes \chi$ for some Hecke character $\chi$. This gives
\begin{equation}\label{ext-sq}
 \Lambda^2(\pi_1 \boxtimes \pi_2) = \Lambda^2({\rm Sym}^3{(\pi_3)} \otimes \chi)
\end{equation}

Let $\omega_{\pi_{i}}$ denote the central character of $ \pi_{i}$ for $i =1,2,3$. From (\cite{asgari-raghuram}, Proposition 3.1), \cite{kim}, we have the following isobaric decompositions of the exterior square of representations.
\begin{equation}\label{ext-sq-LHS-RHS}
\begin{split}
 \Lambda^2(\pi_1\boxtimes \pi_2) &~ =~ ({\rm Sym}^2{(\pi_1)}\otimes \omega_{\pi_2}) \text{ } \boxplus \text{ } ({\rm Sym}^2{(\pi_2)}\otimes \omega_{\pi_1}) \\
    \Lambda^2({\rm Sym}^3{(\pi_3)} \otimes \chi) &~ =~ ({\rm Sym}^4{(\pi_3)} \otimes \chi^2\omega_{\pi_3}) \text{ } \boxplus \text{ } \chi^2\omega_{\pi_3}^3
    \end{split}
\end{equation}

Combining the two equations \eqref{ext-sq} and \eqref{ext-sq-LHS-RHS}, we get
\begin{equation}~\label{isobaric}
{\rm Sym}^2{(\pi_1)}\otimes \omega_{\pi_2} \otimes (\chi^2\omega_{\pi_3}^3)^{-1} \text{ } \boxplus \text{ } {\rm Sym}^2{(\pi_2)}\otimes \omega_{\pi_1} \otimes (\chi^2\omega_{\pi_3}^3)^{-1} ~
 = ~  {\rm Sym}^4{(\pi_3)} \otimes \omega_{\pi_3}^{-2}  \text{ } \boxplus \text{ } 1
\end{equation}

We assumed that ${\rm Sym}^3(\pi_3)$ is cuspidal. If ${\rm Sym}^4(\pi_3)$ is not cuspidal, then using Langlands-Tunnel theorem (see \cite[Thm.1.3, Gelbart's article]{Gelbart}), we conclude that $\pi_3$ must have come from a cusp form of weight one and this already contradicts the assumption on the weight for $\pi_3$ i.e., $k_3 \geq 2$. Thus ${\rm Sym}^4(\pi_3)$ is cuspidal. \smallskip

Thus the first term ${\rm Sym}^4{(\pi_3)} \otimes \omega_{\pi_3}^{-2}$ on RHS in ~\eqref{isobaric} is a representation of $\GL_5(\A_\Q)$ and it will not split further into an isobaric sum with respect to lower dimensional $\GL_{n'}$. We get a contradiction as there is no $\GL_5$-isobaric term on the LHS of Eq.~\eqref{isobaric}. \smallskip
\end{proof}

\begin{rem} The authors of this manuscript established lower bound estimates on the contribution of symmetric cube transfer representations to the cohomology of the $\GL_4$  in \cite{chandrasheel-sudipa}. Prop. \ref{no-overlap} establishes  `lack of overlap between symmetric cube and automorphic tensor product' and hence the estimates obtained in this manuscript improve the lower bound on the total contribution of cuspidal automorphic representations to the cohomology of the $\GL_4$. The case of Asai transfer from ${\GL_{2}} |_ {\K}$ for some imaginary quadratic extension $\K$ over $\Q$ is remaining to be explored. One can ask  these questions for the $\GL_4$ over a general number field and it will be interesting to know whether the similar estimates hold in that case also.
\end{rem}

{\it Acknowledgements:} 
Both the authors are thankful to A. Raghuram for valuable discussions regarding various aspects of this manuscript, in particular about the proof of  `lack of overlap between symmetric cube and automorphic tensor product'. Chandrasheel Bhagwat would like to acknowledge the support of MATRICS research grant MTR/2018/000102 from the Science and Engineering Research Board, Department of Science and Technology, Government of India during this work. Sudipa Mondal would like to acknowledge the support of CSIR PhD fellowship during this work.

\end{document}